\documentclass[11pt]{article}
\usepackage{calrsfs, amsmath, amscd, amsfonts, amssymb, amsthm, tikz, times, bbm, bm}
\usepackage[T1]{fontenc}
\usepackage{listings}
\usepackage{authblk}
\usepackage{xy}
\xyoption{all}

\newcommand\independent{\protect\mathpalette{\protect\independenT}{\perp}}
\def\independenT#1#2{\mathrel{\rlap{$#1#2$}\mkern2mu{#1#2}}}

\newtheorem{thm}{Theorem}[section]

\newtheorem{cor}[thm]{Corollary}
\newtheorem{lem}[thm]{Lemma}
\newtheorem{prop}[thm]{Proposition}
\theoremstyle{definition}
\newtheorem{defi}[thm]{Definition}
\newtheorem*{remark}{Remark}

\newtheorem{example}[thm]{Example}

\def\bi{\begin{itemize}}
\def\ei{\end{itemize}}
\def\aa{\alpha}
\def\bb{\beta}

\definecolor{drab}{rgb}{0.59,0.44,0.09}

\def\be{\begin{equation}}
\def\ee{\end{equation}}
\def\bea{\begin{eqnarray}}
\def\eea{\end{eqnarray}}

\def\({\left(}
\def\){\right)}
\def\[{\left[}
\def\]{\right]}

\def\d={\buildrel d \over =}

\def\dom{\mathit{Dom}}

\def\N{\mathbb{N}}

\def\Z{\mathbb{Z}}

\def \E{\mathbf{E}}
\def \P{\mathbf{P}}
\def \bS{\mathbf{S}}
\def \bU{\mathbf{U}}
\def \bW{\mathbf{W}}
\def \bX{\mathbf{X}}
\def \bY{\mathbf{Y}}

\def \AA{{\cal A}}

\def \FF{{\cal F}}
\def \GG{{\cal G}}

\def \TT{{\cal T}}

\def \XX{{\cal X}}
\def \YY{{\cal Y}}

\newcommand{\commentout}[1]{}

\begin{document}
\title{A de Finetti-type representation of joint hierarchically exchangeable arrays on DAGs}

\author{Jiho Lee}

\affil{Department of Mathematical Sciences, KAIST}

\maketitle

\begin{abstract}
 We define joint exchangeability on arrays indexed by a vector of natural numbers with coordinates being the vertices of directed acyclic graphs (DAGs) using local isomorphisms. The notion provides a new version of exchangeability, which is a joint version of hierarchical exchangeability defined in Jung, L., Staton, Yang (2020). We also prove the existence of a generic representation by independent uniform random variables.

\end{abstract}

\vspace{5mm}


\section{Introduction}\label{sec: intro}
\textit{DAG exchangeability} is a notion of exchangeability on a family of indexed random elements $$\bX=(X_\aa : \aa \in \N^V)$$ on a Borel space $\XX$, where $G=(V,E)$ is a directed acyclic graphs (DAGs). DAG exchangeability was introduced by \cite{jung2020generalization} as a generalization of \textit{hierarchical exchangeability} in \cite{austin:panchenko:2014}. The main purpose of this paper is to extend \cite{jung2020generalization} to a wider class of exchangeable structures including \textit{jointly exchangeable arrays} using probabilistic methods. These methods were first deployed by David Aldous in \cite{aldous1981representations}. Later, Olav Kallenberg applied this method in a systematic way for more general results (see \cite{kallenberg1989representation} or \cite{kallenberg1992symmetries} for example). All of these results are organized in his textbook \cite{kallenberg2006probabilistic} which is a standard reference for fundamental results in exchangeability.

Our work is motivated by studies on Bayesian inference modeling, probabilistic programming, and neural networks as discussed in the introduction to \cite{jung2020generalization}. In fact, the original motivation and hope in that work was to obtain a representation for {\it jointly} DAG exchangeable arrays as opposed to the representation obtained there for {\it separately} DAG exchangeable arrays (precise definitions are given later). In this work, we close this gap by providing such a representation. Briefly, the idea is that de Finetti-type representations of hierarchically exchangeable structures can identify when a hierarchical generative model can be replaced by an equivalent one but with more explicit independence structure (see \cite{statonetal2017pps}). For general applications of exchangeability theory, one can also find in \cite{orbanz2014bayesian} a recent survey on various applications of exchangeability theory to Bayesian inference models including \cite{hoff2008modeling}, \cite{fortini2012predictive}, and \cite{lloyd2012random}. Structure theorems on exchangeable processes also provide canonical representations of neural networks with hierarchical symmetries. Readers can consult, for example, \cite{bloem2019probabilistic}, \cite{cohen2016group}, and \cite{bruna2013spectral} for applications in this direction.

Let $G=(V,E)$ be a DAG. We assume for the rest of the paper that $G$ is finite and simple. Also, when we write $G$ as a set, we refer to the set of vertices $V$. We write $v \prec w$ if there exists a directed nonempty path from $v$ to $w$. Note that $\prec$ defines a partial order in $G$. Conversely, given a finite partially ordered set $(G, \prec)$, we can build a corresponding set of directed edges $E$ by adding the edge $\overrightarrow{vw}$ if and only if $v \prec w$ and there is no $v' \in G$ such that $v \prec v' \prec w$. To make this correspondence bijective, we assume that $G$ always have the minimal set of edges under its induced partial order: that is, whenever there is a directed path from $v$ to $w$ that passes other vertices than $v$ and $w$, we have no edge from $v$ to $w$. 

We say that a subgraph $C$ of a DAG $G$ is \textbf{downward-closed} (or just closed) if it is downward-closed under the induced partial order, that is, $v \in C$ whenever there exists $w \in C$ such that $v \prec w$. We write $\AA_C$ for the collection of all closed subgraphs of $C$. For $\aa \in \N^G$ and $C \in \AA_G$, let $\aa|_C$ denote the restriction of $\aa$ to $C$ when viewing $\aa$ as a function from $G$ to $\N$.

\begin{defi}\label{def: sep DAG exch} Let $G$ be a DAG. Then, a $G$-\textbf{permutation}\footnote{Although in \cite{jung2020generalization} we used the word ``automorphism,'' we change the terminology in order to distinguish them with automorphisms of the DAG itself which appear in Section \ref{sec: joint DAG}.)} is a bijection $\tau : \N^G \to \N^G$ such that 
\begin{equation}\label{eq: G automorphism} \aa|_C = \bb|_C \Longleftrightarrow \tau(\aa)|_C = \tau(\bb)|_C
\end{equation}
for all $\aa, \bb \in \N^G$, $C \in \AA_G$. We write $ S_\N^{\rtimes G}$ for the collection of all $G$-permutations. A random array $\bX=(X_\aa : \aa \in \N^G)$ is \textbf{DAG-exchangeable} if for all $\tau \in S_\N^{\rtimes G}$, we have
\begin{equation}\label{eq: sep DAG exchangeability} \Big( X_\alpha : \alpha \in \N^G \Big) \buildrel d \over = \Big( X_{\tau(\alpha)} : \alpha \in \N^G \Big).
\end{equation}
\end{defi}

For $C \in \AA_G$, let $I_C := \underset{C' \in \AA_C}{\bigcup} \N^{C'}$. Given a $G$-permutation $\tau$ and $C \in \AA_G$, one can always define the action of $\tau$ on $\N^C$ by $\tau (\alpha) (v) = \tau (\beta) (v)$ for any $\beta \in \N^G$ such that $\beta|_C = \alpha$, since by the definition of $G$-permutations the choice of $\beta$ is irrelevant. Since this induced map on $\N^C$ is also a $C$-permutation, we can regard $\tau$ as a bijection from $I_G$ into itself, and identify it as a $G$-permutation of the index set $I_G$. Therefore, we can define DAG exchangeability on a random array indexed by $I_G$ instead of $\N^G$. For $\aa \in I_G$, let $\dom(\aa)$ be the domain of the function $\aa$, and $Restr(\aa):=\{\alpha|_C : C \in \AA_{\dom(\aa)}\}.$

For any DAG-exchangeable array, we have a canonical representation using independent uniform random variables, as long as the underlying probability space is rich enough. We will assume this condition for the rest of the paper.

\begin{thm}[\cite{jung2020generalization}]\label{thm: separate DAG, 1} Let $G$ be a DAG. Let $\bX=(X_\aa : \aa \in \N^G)$ be a DAG-exchangeable array taking values in a Borel space $\XX$. Then, there exist a measurable function $f:[0,1]^{\AA_G} \to \XX$ and an i.i.d. array $\bU=(U_\bb : \bb \in I_G)$ of uniform random variables such that
\begin{equation}\label{eq: Smain 1} X_\aa = f\Big(U_{\bb} : \bb \in Restr(\aa)\Big)
\end{equation}
almost surely for all $\aa \in \N^G$. 
\end{thm}

\begin{example}\label{example: past} This setup covers the following past results on the representations of exchangeable structures by independent uniform random variables.
\begin{itemize}
    \item [(a)] {\it Exchangeable sequences:} Let $G$ be a graph with a single vertex. Then, $S_\N^{\rtimes G}$ is simply the group of all bijections from $\N$ to itself. So, a DAG-exchangeable array is merely an exchangeable sequence. Theorem \ref{thm: separate DAG, 1} implies that for an exchangeable sequence $\bX=(X_n : n \in \N)$, there exist an i.i.d. sequence of uniform random variables $(U_0, U_1, U_2,...)$ and a measurable function $f:[0,1]^2 \to \XX$ such that
    $$ X_n = f(U_0, U_n)$$
    almost surely for all $n \in \N$. This is a variant of de Finetti's theorem (\cite{de1929funzione}, \cite{de1937prevision}, \cite{hewitt1955symmetric}) proposed by \cite{aldous1981representations}.
    \item [(b)] {\it Separately exchangeable arrays:} Let $G=(\{1,2\}, \emptyset).$ Then, $\N^G=\N^2$ and $S_\N^{\rtimes G}$ is isomorphic to $(S_\N)^2$, acting naturally on $\N^2$. Thus, a DAG-exchangeable array is a separately exchangeable array of dimension $2$, that is, it satisfies the distributional equation
    \begin{equation}\label{eq: separate AH} (X_{ij} : i,j \in \N) \buildrel d \over = (X_{\tau(i)\rho(j)} : i,j \in \N).
\end{equation}
    It is guaranteed by either Theorem \ref{thm: separate DAG, 1} or the  Aldous-Hoover theorem (\cite{aldous1981representations}, \cite{hoover1979relations}) that there exist an i.i.d. array of uniform random variables $\bU=(U_{ij} : i,j \geq 0)$ and a measurable function $f: [0,1]^4 \to \XX$ such that
    \begin{equation}\label{eq: sep exch rep}
    X_{ij}=f(U_{00}, U_{i0}, U_{0j}, U_{ij})    
    \end{equation}
    almost surely for all $i,j \in \N$. The result can be extended to arrays of higher dimensions. (See \cite{kallenberg2006probabilistic} for a deep analysis on exchangeable arrays of high dimensions.)
    \item [(c)]{\it Hierarchically exchangeable arrays:} Let $G=(\{v_1, v_2, u_1, u_2\}, E)$ where $E=\{\overrightarrow{v_1 v_2},\overrightarrow{u_1 u_2}\}$. Then, a DAG-exchangeable array is an example of hierarchical exchangeability introduced by \cite{austin:panchenko:2014}, which can be written in the form $(X_{ij,k\ell}: i,j,k,\ell \in \N)$ where $i,j,k,\ell$ are the coordinates on $v_1, v_2, u_1, u_2$, respectively. If an array $\bX$ indexed by $\N^G$ is DAG-exchangeable, then for $\tau, \rho \in S_\N$ and $\tau_i, \rho_k \in S_\N$ for each $i,k \in \N$, we have
    $$(X_{ij,k\ell} : i,j,k,\ell \in \N) \buildrel d \over = (X_{\tau(i)\tau_i(j),\rho(k)\rho_k(\ell)} : i,j,k,\ell \in \N).$$
    The representation theorem by \cite{austin:panchenko:2014} allows us to have
    $$X_{ij,k\ell} = f(U_{00,00},U_{i0,00},U_{ij,00},U_{00,k0},U_{i0,k0},U_{ij,k0},U_{00,k\ell},U_{i0,k\ell}, U_{ij,k\ell}).$$
    almost surely for some measurable function $f$ and some i.i.d. array $\bU$ of uniform random variables.
\end{itemize}
\end{example}

\begin{figure}[ht]
	\centering
	\includegraphics[height=2.4in]{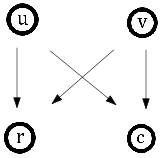}
	\caption{The DAG for Example \ref{example: block matrix}.}
	\label{fig:DAG:random-block-matrices}
\end{figure}

\begin{example}\label{example: block matrix} We introduce \textit{random block matrices} from Example 2.2 of \cite{jung2020generalization}, which is a new example covered by Theorem \ref{thm: separate DAG, 1}. Let $V=\{u,v,r,c\}$, $E=\{\overrightarrow{ur},\overrightarrow{uc},\overrightarrow{vr},\overrightarrow{vc} \}.$ (See Figure \ref{fig:DAG:random-block-matrices}.) An array $(X_{ij,k\ell} : i,j,k,l \in \N)$ is DAG-exchangeable (regarding $i,j,k,\ell$ as coordinates on $u,v,r,c,$ respectively) if for all $\tau, \rho \in S_\N$ and $\tau_{ij}, \rho_{ij} \in S_\N$ with $i,j \in \N$, we have
\begin{equation}\label{eq: random BM sep} (X_{ij,k\ell} : i,j,k,l \in \N) \buildrel d \over = (X_{\tau(i)\rho(j),\tau_{ij}(k)\rho_{ij}(\ell)} : i,j,k,l \in \N).
\end{equation}

By Theorem \ref{thm: separate DAG, 1}, there exist a measurable function $f$ and an i.i.d. array of uniform random variables $\bU$ such that for all $i,j,k,\ell \in \N$, we have
\begin{equation}\label{eq: random BM sep representation} X_{ij,k\ell}=f(U_{00,00},U_{i0,00},U_{0j,00},U_{ij,00},U_{ij,k0},U_{ij,0\ell},U_{ij,k\ell})
\end{equation}
almost surely.
\end{example}
To motivate the main objective of this paper, let us revisit (b) of Example \ref{example: past}. Let us consider the case where the array is jointly exchangeable, that is, \begin{equation}\label{eq: joint AH} (X_{ij} : i,j \in \N) \buildrel d \over = (X_{\tau(i)\tau(j)} : i,j \in \N)
\end{equation}
for all $\tau \in S_\N$. This is weaker than separate exchangeability, where we can choose permutations on the two coordinates separately. Jointly exchangeable arrays of dimension two are, in particular, closely connected to random graph theory. We recommend \cite{diaconis2008graph} or \cite{austin2008exchangeable} as an introduction towards this direction.

For jointly exchangeable arrays, we have a representation of the form
\begin{equation}\label{eq: joint AH rep} 
X_{ij}=f(U_0, U_i, U_j, U_{\{i,j\}})
\end{equation}
almost surely for $i \neq j$ (\cite{hoover1979relations}). One can see that, compared to \eqref{eq: sep exch rep}, the indices on the rows are merged with those on the columns. We can naturally ask if the similar merging occurs on joint versions of DAG-exchangeable arrays. That is, if a random array $\bX=(X_{ij,k\ell} : i,j,k,\ell \in \N)$, for instance, satisfies the distributional equation
\begin{equation}\label{eq: BM joint 1} (X_{ij,k\ell} : i,j,k,\ell \in \N) \buildrel d \over = (X_{\tau(i)\tau(j),\tau_{\{i,j\}}(k)\rho_{\{i,j\}}(\ell)} : i,j,k,\ell \in \N)
\end{equation}
for all $\tau, \tau_{\{i,j\}},\rho_{\{i,j\}} \in S_\N$, we can ask whether we have a representation of the form
\begin{equation}\label{eq: BM joint 1 rep}
X_{ij,k\ell}=f(U_{0,00},U_{i,00},U_{j,00}, U_{\{i,j\},00}, U_{\{i,j\},k0}, U_{\{i,j\},0\ell}, U_{\{i,j\},k\ell})
\end{equation}
almost surely for $i \neq j$.

The main objective of this paper is to extend the representation given by Theorem \ref{thm: separate DAG, 1} to a wider class of exchangeable structures. This new model includes Hoover's joint exchangeable arrays, the representation \eqref{eq: BM joint 1 rep}, and exchangeable arrays associated to arbitrary DAGs with merging of the vertices in the sense described above. We will rigorously define the model in the next section with more examples.

\section{Settings and Main Results}\label{sec: joint DAG}
\subsection{Main Results}
Let $Aut(G)$ denote the directed graph automorphism group of $G$, and let $K$ be a subgroup of $Aut(G)$. Define a left group action of $Aut(G)$ acting on $I_G$ by
$$ \kappa \beta (v) = \beta (\kappa^{-1}(v)), \ \kappa \in Aut(G), \ \beta \in I_G.$$

Since elements in both $Aut(G)$ and $S_\N^{\rtimes G}$ acts as bijections from $I_G$ to itself, we can composite them as functions. For $\kappa \in Aut(G)$ and $\tau \in S_\N^{\rtimes G}$, we write $\kappa \tau$, $\tau \kappa$ to denote composite functions $\kappa \circ \tau$, $\tau \circ \kappa : I_G \to I_G$, respectively. 

As we can see in Example \ref{example: past}, in the case of a separately exchangeable array of dimension $d$, we can regard $G$ as a graph of order $d$ with no edges. An array $\bX = (X_\alpha : \alpha \in \N^d)$ is separately exchangeable if and only if 
\begin{equation}\label{eq: AH dimension d}
    (X_\alpha : \alpha \in \N^d) \buildrel d \over = (X_{\tau \alpha} : \alpha \in \N^d)
\end{equation} 
for all $\tau \in S_\N^d$. 

On the other hand, the joint exchangeability can be described in terms of commutativity with the automorphism group $Aut(G)$, which is isomorphic to $S_d$. The array $\bX$ is jointly exchangeable, or equivalently,
\begin{equation}\label{eq: joint AH high dimension}
(X_{n_1 n_2 \cdots n_d} : n_1,...,n_d \in \N) \buildrel d \over = (X_{\sigma(n_1) \sigma(n_2) \cdots \sigma(n_d)} : n_1,...,n_d \in \N)
\end{equation}
for all $\sigma \in S_\N$ if and only if \eqref{eq: AH dimension d} holds for all $\tau \in S_\N^d$ satisfying $\tau \kappa = \kappa \tau$ for any $\kappa \in Aut(G)$.

Furthermore, if we require that $\tau$ commutes with $\kappa \in K$ for subgroups $K$ of $S_d$ instead of the whole $S_d$, we obtain a different notion of exchangeability lying between separate and joint exchangeability. For example, if we let $d=3$ and $K$ be the subgroup of $S_3$ generated by the transposition $(2 \ 3)$,  we have
$$(X_{ijk} : i,j,k \in \N^3) \buildrel d \over = (X_{\tau(i)\sigma(j)\sigma(k)}: i,j,k \in \N^3)$$
for all $\tau, \sigma \in S_\N$.
\begin{figure}
    \centering
    \includegraphics[width=70mm,height=50mm]{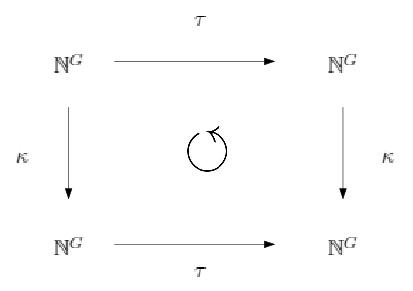}
    \caption{A diagram associated to permutations commuting with graph automorphisms. $\tau \in S_\N^{\rtimes G}$ commutes with all $\kappa \in K \subseteq Aut(G)$.}
    \label{fig:CD1}
\end{figure}

Considering the above observations, it is tempting to define joint exchangeability on random arrays defined on DAGs by assigning a subgroup $K$ of $Aut(G)$ and allowing law-invariance for permutations which commute with $K$ (see Figure \ref{fig:CD1}). However, there are a few issues we have to handle. One is that the exchangeability structure does not uniquely determine the group $K$.

\begin{example}\label{example: nonuniqueness} let $G=\{v_1, v_2, v_3\}$ be a graph with no edges, and let $K=\Z/3\Z$ acting naturally on $G$. Then, since $K$ acts transitively on $G$, a permutation $\tau=(\tau_1,\tau_2,\tau_3) \in S_\N^3$ of $\N^G$ commutes with $K$ if and only if $\tau_1=\tau_2=\tau_3$. So we obtain the same exchangeability structure in this setting if we choose either $K=\Z/3\Z$ or $K=S_3$. 
\end{example}

The other issue is more serious. Many of the proofs of representation theorems on exchangeable arrays use induction on the dimension of the arrays, and we will also follow this strategy. However, by restricting $K$ to be a subgroup of $Aut(G)$, we encounter an issue when deploying this type of induction, as the following example shows.

\begin{example}\label{example: bad restriction} Let $V=\{u_1,u_2,v_1,v_2\}$, $E=\{\overrightarrow{u_1 u_2}, \overrightarrow{v_1 v_2}  \}$ (Example \ref{example: past}, (c)). Then, $Aut(G)$ is a group of order two, where the nonidentity element exchanges $u_i$ and $v_i$ for $i=1,2$ respectively. However, the closed subgraph $C=\{u_1, u_2, v_1\}$ has a trivial automorphism group. If we assign joint exchangeability on a random array $\bX = (X_\alpha : \alpha \in \N^G)$ associated to $K=Aut(G)$, the permutations in consideration should act identically on vertices $u_1$ and $v_1$. However, there is no way to assign such a class of exchangeability on random arrays defined on the subgraph $C$ via its automorphism group, since it has no nontrivial graph automorphism at all.
\end{example}

Both of these issues arise from the nature that the class of permutations that commute with $K$ is determined only by the \textit{local} behavior of $K$ in the following sense. Let $Z_K(S_\N^{\rtimes G})$ denote the group of $G$-permutations that commute with $K$. Let $C_v$ denote the closure of $\{v\}$, i.e. the smallest closed subgraph of $G$ containing $v$. Then, $Z_K(S_\N^{\rtimes G})$ consists of all $\tau \in S_\N^{\rtimes G}$ such that $\tau\kappa (\alpha)= \kappa \tau (\alpha)$ for all $v \in G$, $\kappa \in K$, $\alpha \in \N^{C_v}$. In other words, the only relevant information from $K$ is its action on ``local'' indices, that is, elements in $I_G$ whose domain is of the form $C_v$.

Now, in order to handle the above issues from the examples, instead of a subgroup of $Aut(G)$, we will use a collection of mappings which takes into account the local nature required by joint exchangeability. We will continue to use $K$ to denote such a collection. These mappings are not necessarily defined on the whole of $G$, but only on a specific vertex and its closure. We use the word \textbf{isomorphism} to describe a bijective function from a closed subgraph of $G$ to another closed subgraph that preserves directed edges. 

\begin{defi}\label{def: CLIC} Let $G$ be a DAG. A \textbf{local isomorphism} of $G$ is a sub-DAG isomorphism of the form $$\kappa : C_v \to C_w$$ for some $v, w \in G$. 

A collection $K$ of local isomorphisms is called a \textbf{consistent local isomorphism class (CLIC)} of $G$ if
\begin{itemize}
    \item $K$ contains all the identity mappings and is closed under inversion, composition, and restrictions to subgraphs of the form $C_u$. 
    \item If $\kappa$ is a local isomorphism such that for each $v \in \dom(\kappa)$ we have $\kappa' \in K$ such that $\kappa'(v)=\kappa(v)$, then $\kappa \in K$.
\end{itemize}
\end{defi}

\begin{example}\label{example: CLIC} Let us go back to the random block matrices in Example \ref{example: block matrix}. The following are the list of all the local isomorphisms of $G$:

\begin{itemize}
    \item $\kappa_1 : \{u\} \to \{v\}$.
    \item $\kappa_{01} : \{u, v, r\} \to \{u, v, c\}$ where $\kappa_{01}(u)=u$, $\kappa_{01}(v)=v$, and $\kappa_{01}(r)=c$.
    \item $\kappa_{11} : \{u, v, r\} \to \{u, v, c\}$ where $\kappa_{11}(u)=v$, $\kappa_{11}(v)=u$, and $\kappa_{11}(r)=c$.
    \item $\kappa_{1r} : \{u,v,r\} \to \{u,v,r\}$ where $\kappa_{1r}(u)=v$, $\kappa_{1r}(v)=u$, and $\kappa_{1r}(r)=r$.
    \item $\kappa_{1c} : \{u,v,c\} \to \{u,v,c\}$ where $\kappa_{1c}(u)=v$, $\kappa_{1c}(v)=u$, and $\kappa_{1c}(c)=c$.
    \item Inverses of the above maps
    \item Identities
\end{itemize}

The following are all the possible lists of members a CLIC can have, where the identities and the inverses are omitted:
\begin{enumerate}
    \item $\kappa_1, \kappa_{1r}, \kappa_{1c}$
    \item $\kappa_{01}$
    \item $\kappa_1, \kappa_{1r}, \kappa_{1c},  \kappa_{01}, \kappa_{11}$ 
\end{enumerate}

For another example, let us consider $G=(V,E)$ with $V=\{u_1,u_2,v_1,v_2,v_3\}$ with edges $E=\{\overrightarrow{u_1 u_2}, \overrightarrow{v_1 v_2}, \overrightarrow{v_2 v_3}\}$. It corresponds to Austin and Panchenko's setting with $r=2$, $d_1=2$, $d_2=3$ (See \cite{austin:panchenko:2014}). Although there is no nontrivial automorphism of $G$, we have the following nontrivial local isomorphisms along with their inverses:
\begin{itemize}
    \item $\rho_1 : \{u_1\} \to \{v_1\}$.
    \item $\rho_2 : \{u_1, u_2\} \to \{v_1,v_2\}$ where $\rho_2(u_i)=v_i$.
\end{itemize}

The following are all the possible lists of members a CLIC can have, where again the identities and the inverses are omitted:
\begin{enumerate}
    \item[4.] $\rho_1 $ 
    \item[5.] $\rho_1, \rho_2$ 
\end{enumerate}
\end{example}

\begin{remark} A local isomorphism need not be extendable to an automorphism. For instance, in the case of Example \ref{example: bad restriction}, the cause of the second issue is that the local isomorphism $\kappa : \{u_1\} \to \{v_1\}$ cannot be extended to an automorphism of $C$. Let $K=\{id, \kappa, \kappa^{-1}, \rho, \rho^{-1}\}$, where $\rho : \{u_1, u_2\} \to \{v_1, v_2\} $ with $\rho(u_i)=v_i$ for $i=1,2$. Then, the (jointly) exchangeable random array associated to the automorphism group of $G$ is law-invariant under the permutations that commute ``locally'' with $K$. Unlike the case using automorphisms, the induced symmetry on the subgraph $C$ is well-described by just taking the elements in $K$ which are defined inside $C$, which are $id, \ \kappa,$ and $\kappa^{-1}$.
\end{remark}

 Let $K$ be a CLIC. For $v \in G$, let $K_v$ be the collection of $\kappa \in K$ such that $C_v \subseteq \dom(\kappa)$. We say that two vertices $v,w \in G$ are \textbf{equivalent} under $K$ if there exists $\kappa \in K_v$ such that $\kappa(v)=w$, and denote this relation by $v \buildrel K \over \sim w$. 
 
 We can define a similar equivalence in $I_G$ as well. Given $v \in G$, $\kappa \in K_v$ and $\aa \in \N^{C_v}$, define $\kappa (\aa) \in \N^{C_{\kappa(v)}}$ as
 $$ \kappa (\aa) (u) := \aa( \kappa^{-1} u )$$
 for $u \in C_{\kappa(v)}$. We say that two indices $\alpha, \beta \in I_G$ are \textbf{equivalent} under $K$ if there exists a bijection $\phi: \dom(\alpha) \to \dom(\beta)$ such that for each $v \in \dom(\alpha)$, there exists $\kappa \in K_v $ such that $\kappa(\alpha|_{C_v})=\beta|_{C_{\phi(v)}}$. We also write $\alpha \buildrel K \over \sim \beta$ for this relation. It is easy to check that both the relations on $G$ and $I_G$ denoted by the symbol $\buildrel K \over \sim$ are equivalence relations.

It is convenient to have our index set to be transitive under the group action. Thus, instead of $\N^G$, we restrict our index set to $${\N}_K^G:=\{\alpha \in \N^G : \kappa(\alpha|_{C_v}) \neq \alpha|_{C_{\kappa(v)}}  \text{ for all } v \in G, \ \kappa \in K_v \text{ such that } \kappa|_{C_v} \neq id_{C_v}\}.$$ Let us also write $$I_K^G:=\underset{C \in \AA_G}{\bigcup} \N_K^C.$$

Roughly speaking, a consistent isomorphism class $K$ is an indicator that restricts the permutations of interest to act identically on vertices that are equivalent under $\buildrel K \over \sim$. The inclusion of identities and taking closure under inversion, composition and restriction, has ensured that $\buildrel K \over \sim$ is an equivalence relation in both $V$ and $I_K^G$.

Now we are ready to define joint DAG-exchangeability and state the main theorem. 

\begin{defi}\label{def: joint DAG exc} Let $G$ be a DAG and $K$ a CLIC of $G$. A permutation $\tau \in S_\N^{\rtimes G}$ is said to be $K$\textbf{-commuting} if for all $\beta \in \N_K^{C_v}$ with $v \in G$ and $\kappa \in K_v$, we have
\begin{equation}\label{eq: commutativity with K}
\tau \circ \kappa(\beta)=\kappa \circ \tau(\beta).
\end{equation}

An array $\bX=(X_\alpha : \alpha \in \N_K^G)$ (or $I_K^G$) is $(G, K)$\textbf{-exchangeable} if
\begin{equation}\label{eq: joint DAG exchangeability}
(X_\alpha : \alpha \in \N_K^G) \buildrel d \over = (X_{\tau\alpha} : \alpha \in \N_K^G)
\end{equation}
for all $K$-commuting $\tau$. 
\end{defi}
\begin{figure}
    \centering
    \includegraphics[width=70mm,height=50mm]{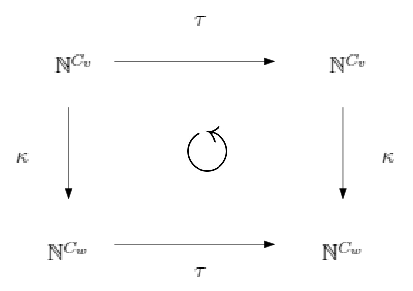}
    \caption{A commutative diagram associated to Definition \ref{def: joint DAG exc}.}
    \label{fig:CD2}
\end{figure}

We will keep using the notation $Z_K(S_\N^{\rtimes G})$ for the collection of all $K$-commuting permutations. Note that Definition \ref{def: sep DAG exch} is a special case of Definition \ref{def: joint DAG exc} where $K$ consists only of identity mappings. Also, thanks to the second condition in Definition \ref{def: CLIC}, for any two CLIC's $K_1$ and $K_2$ on a DAG $G$, we have $K_1=K_2$ whenever $Z_{K_1}(S_\N^{\rtimes G})=Z_{K_2}(S_\N^{\rtimes G})$. 

\begin{thm}\label{thm:main} Let $G$ be a finite DAG, $K$ a CLIC of $G$. Then, an array $\bX=(X_\alpha : \alpha \in \N_K^G)$ is $(G,K)$-exchangeable if and only if there exists a measurable function $f: [0,1]^{\AA_G} \to \XX$ such that for all $\alpha \in {\N}_K^G$,
\begin{equation}\label{eq: main1}
X_\alpha \buildrel a.s. \over = f \big( U_{[\alpha|_C]_K} : C \in \AA_G \big)
\end{equation}
for some array $\bU$ of i.i.d. uniform random variables indexed by $([\beta]_K : \beta \in I_K^G)$, where $[\beta]_K$ denotes the equivalence class of $\beta$ with respect to $\buildrel K \over \sim$. 
\end{thm}

\begin{example} Let us inspect the classes of permutations associated to the CLIC's introduced in Example \ref{example: CLIC} and the representations of the according exchangeable arrays. For the first case (random block matrices), each of the describe CLIC's represents the law-invariance under the following permutations of indices, respectively, where $i,j,k,\ell$ are the index values at $u,v,r,c$, respectively:
\begin{enumerate}
    \item $X_{ij,k\ell} \to X_{\tau(i)\tau(j),\rho_{\{i,j\}}(k)\lambda_{\{i,j\}}(\ell)}$
    \item $X_{ij,k\ell} \to X_{\tau(i)\theta(j),\rho_{ij}(k)\rho_{ij}(\ell)}$
    \item $X_{ij,k\ell} \to X_{\tau(i)\tau(j),\rho_{\{i,j\}}(k)\rho_{\{i,j\}}(\ell)}$
\end{enumerate}

For the three cases, Theorem \ref{thm:main} provides representations of the following forms, respectively:
\begin{enumerate}
    \item $X_{ij,k\ell}=f(U_{0,00},U_{i,00},U_{j,00}, U_{\{i,j\},00}, U_{\{i,j\},k0}, U_{\{i,j\},0\ell}, U_{\{i,j\},k\ell}), \ i\neq j.$
    \item $X_{ij,k\ell}=f(U_{00,0},U_{i0,0},U_{0j,0}, U_{ij,00}, U_{ij,k}, U_{ij,\ell}, U_{ij,\{k,\ell\}}),\ (i,j,k) \neq (i,j,\ell).$
    \item $X_{ij,k\ell}=f(U_{0,0},U_{i,0},U_{j,0}, U_{\{i,j\},0}, U_{\{i,j\},k}, U_{\{i,j\},\ell}, U_{\{i,j\},\{k,\ell\}}),\ i\neq j,\ (i,j,k) \neq (i,j,\ell).$
\end{enumerate}
For the second example, each of the cases allows permutations of the following forms, respectively, where $i,j,k,\ell,m$ are the index values at $u_1,v_1,u_2,v_2,v_3,$ respectively:
\begin{enumerate}
    \item[4.] $X_{ij,k\ell, m} \to X_{\tau(i)\tau(j),\rho_{i}(k)\lambda_{j}(\ell) \theta_{j \ell} (m)}$
    \item[5.] $X_{ij,k\ell, m} \to X_{\tau(i)\tau(j),\rho_{i}(k)\rho_{j}(\ell) \theta_{j \ell} (m)}$
\end{enumerate}

For each of the two cases, Theorem \ref{thm:main} provides a representation of the following forms:
\begin{enumerate}
    \item[4.] $\begin{aligned}[t] X_{ij,k\ell,m} = f( & U_{0,00,0}, U_{i,00,0}, U_{i,k0,0}, U_{j,00,0}, U_{\{i,j\},00,0}, U_{\{i,j\},k0,0}, \\ & U_{j,0\ell,0}, U_{\{i,j\},0\ell,0}, U_{\{i,j\},k\ell,0}, U_{j,0\ell,m}, U_{\{i,j\},0\ell,m},U_{\{i,j\},k\ell,m}), \\ & i \neq j. \end{aligned}$
    \item[5.] $\begin{aligned}[t] X_{ij,k\ell,m} = f( & U_{0,0,0}, U_{i,0,0}, U_{i,k,0}, U_{j,0,0}, U_{\{i,j\},0,0}, U_{\{i,j\},k,0}, \\ & U_{j,\ell,0}, U_{\{i,j\},\ell,0}, U_{\{i,j\},\{k,\ell\},0}, U_{j,\ell,m}, U_{\{i,j\},\ell,m},U_{\{i,j\},\{k,\ell\},m}), \\ & i \neq j, \ (i,k) \neq (j,\ell). \end{aligned} $
\end{enumerate}
\end{example}

\subsection{Symmetry random variables associated to jointly DAG-exchangeable arrays}\label{section: symmetry}

The overall plan of the proof of Theorem \ref{thm:main} is similar to that of Theorem \ref{thm: separate DAG, 1}. We deploy induction on the number of vertices of $G$. To do this, we have to construct random variables which encode the intermediate information associated to $\bX$, which we call a \textit{symmetry random variables} associated to $\bX$. We will see that the randomness of the uniform random variables affect $\bX$ only through symmetry random variables. A typical example of this phenomenon is the role of the empirical distribution in an exchangeable sequence (see Lemma 7.1 of \cite{kallenberg2006probabilistic}).

The key property we need to show in this strategy is conditional independence among the involved $\sigma$-fields, and that is Proposition \ref{prop:keytool} in our case. It is a parallel of Proposition 4.1 of \cite{jung2020generalization}, of which the proof is based on results of \cite{hoover1979relations}. One aspect of Hoover's proof is that it depends heavily on nonstandard analysis and symbolic logic. In this paper, we provide a probabilistic proof of Proposition \ref{prop:keytool} independent of Hoover's. As mentioned at the beginning of the paper, our strategy resembles that of \cite{kallenberg2006probabilistic} (especially Chapter 7) in the proof of the Aldous-Hoover representation theorem in a sense that we use systematic tools to prove conditional independence between involved random variables to deploy coding lemmas that provides representations using independent uniform random variables. (See the appendix for the lemmas that we use in the proof of the main result.)

Let $\bX$ be a $(G,K)$-exchangeable array. For $\aa \in I_K^G$, let us write $C_{K,\aa}(S_\N^{\rtimes G})$ for the collection of $K$-commuting permutations $\tau$ such that $\tau \alpha = \alpha$.

The following are basic properties of $C_{K,\aa}(S_\N^{\rtimes G}).$ 
\begin{itemize}
    \item[(a)] If $\alpha \in Restr(\beta)$, then $C_{K,\bb}(S_\N^{\rtimes G}) \subseteq C_{K,\aa}(S_\N^{\rtimes G}).$
    \item[(b)] If $\alpha \buildrel K \over \sim \beta$, then $C_{K,\aa}(S_\N^{\rtimes G})=C_{K,\bb}(S_\N^{\rtimes G}).$
\end{itemize} 
The property (a) is obvious. The new property (b) follows from the fact that $\tau \in Z_K(S_\N^{\rtimes G})$ commutes with the elements of $K$.

Let $\FF_\alpha$ denote the invariant $\sigma$-field of $C_{K,\aa}(S_\N^{\rtimes G})$, that is,
$$\FF_\aa : = \{A \in \sigma(\bX) : \tau A = A \text{ for all } \tau \in C_{K,\aa}(S_\N^{\rtimes G})\}.$$
(we have $\FF_\aa = \FF_\bb$ for $\aa \buildrel K \over \sim \bb$ by (b).) We want to construct a Borel-valued random array $\bS:=(S_\alpha : \alpha \in I_K^G)$ satisfying the following properties, and call it a \textbf{random symmetry array} associated to $\bX$:

\begin{enumerate}
    \item $\FF_\alpha = \sigma(S_\alpha).$
    \item The array $(\bX, \bS)$ is $(G,K)$-exchangeable.
    \item For $\alpha \buildrel K \over \sim \beta$, $S_\alpha=S_\beta.$
\end{enumerate}

The existence of symmetry arrays is a straightforward exercise. We give the proof in Appendix \ref{app: sym array}.

\begin{prop}\label{prop: existence of S joint} For any $(G,K)$-exchangeable array $\bX$ taking values in a Borel space, a random symmetry array exists.
\end{prop}

Once we have an associated symmetry array, we can improve Theorem \ref{thm:main} so that the dependence structure of intermediate $\sigma$-fields is more explicit. For a generic array $\bY=(Y_i : i \in I)$ and $J \subseteq I$, we write $\bY_J :=(Y_i : i \in J).$

Assign a well-ordering on $\AA_G$. For each equivalence class of $I_K^G$ under $\buildrel K \over \sim$, choose a representative whose domain is the smallest under this well-ordering. From now on, let us assume that we have a fixed collection of such representatives, and denote this collection by $\Gamma_K^G$.\footnote{For those who are concerned with using the axiom of choice in this procedure, we note that it is not the case. Since $\AA_G$ is a finite set, we do not need the well-ordering principle when we choose a well-ordering. When choosing the representatives, for each subgraphs we can assign a well-ordering on the vertices and select the smallest element in the lexicographical order.}

\begin{thm}\label{thm:main2} Let $G$, $K$, $\bX$ be as in Theorem \ref{thm:main}, and let $\bS$ be a symmetry array of $\bX$. Then, there exist measurable functions $f_C: [0,1]^{\AA_C} \to \XX$ such that for all $\alpha \in \Gamma_K^G$,
\begin{equation}\label{eq: main2}
S_\alpha \buildrel a.s. \over = f_{\dom(\alpha)} \big( \bS_{Restr'(\gamma_\alpha)}, U_\alpha \big)
\end{equation}
for some array $\bU$ of i.i.d. uniform random variables indexed by $\Gamma_K^G$.
\end{thm}

Note that by recursively replacing $S_\beta$ with $f_{\dom(\beta)}\big( \bS_{Restr'(\gamma_\beta)}, U_\beta \big)$, for each $\aa \in I_K^G$ we obtain the alternate representation of \eqref{eq: main2} of the form
\begin{equation}\label{eq: main3} S_\alpha = g_{\dom(\alpha)}\Big(U_\beta : \beta \in Restr(\gamma_\alpha)\Big)
\end{equation}
for some measurable functions $g_C$, where $U_\bb = U_{\gamma_\bb}$ for $\bb \in I_K^G$.

As mentioned earlier, the basic strategy of our proof is using induction on $|G|$, the number of vertices of $G$. We first build representations on the proper subgraphs of $G$, and tie them all together into a representation in the whole $G$. Proposition \ref{prop:keytool} is a key result which makes this ``tying'' possible.

Let $\alpha, \beta \in I_K^G$, and define $Restr(\alpha,K)$ to be the collection of $\alpha' \in I_K^G$ such that $\alpha' \buildrel K \over \sim \alpha|_C$ for some $C \in \AA_{\dom(\alpha)}$. Let $$D_{\alpha, \beta}:=\{v \in \dom(\alpha) : \alpha|_{C_v}\in Restr(\beta,K)\},$$ and define $$\alpha \wedge \beta := \alpha|_{D_{\alpha,\beta}}.$$ These are the joint-exchangeability counterparts of restrictions and intersections of two indices for separate DAG-exchangeability. 

\begin{remark} We have $\aa \wedge \beta \buildrel K \over \sim \beta \wedge \aa$ for all $\aa, \bb \in I_K^G.$ For each $v \in D_{\aa,\bb}$, there exists $u \in \dom(\bb)$ such that $\aa|_{C_v}=\kappa(\bb|_{C_u})$ for some $\kappa \in K$. By the definition of $I_K^G$, such a vertex $u$ is unique, so we can define a mapping $\phi: D_{\aa,\bb} \to D_{\bb,\aa}$ via this relation. One can easily see that $\phi$ is the desired bijective correspondence to guarantee that $\aa \wedge \bb \buildrel K \over \sim \bb \wedge \aa.$ One can also easily check that $\aa \wedge \bb, \bb \wedge \aa \in Restr(\aa,K).$
\end{remark}  

\begin{prop}\label{prop:keytool} Let $\alpha_1,...,\alpha_n \in I_K^G$. Then, $(S_{\alpha_k}: k\leq n)$ are independent given $(S_{\alpha_k \wedge \alpha_j} : k \neq j).$
\end{prop}

The next corollary follows from Proposition \ref{prop:keytool}. Let us first define some notation.
\begin{itemize}
    \item $Restr'(\alpha,K):= Restr(\alpha,K)\backslash [\alpha]_K$,
    \item $J_k:=\{ \aa \in \Gamma_K^G : |\dom(\aa)| \leq k \}$,
\end{itemize}

\begin{cor}\label{cor: independence given the previous} Let $\bS_k:=(S_\alpha : \aa \in J_k).$ Then, for $1 \leq k \leq |G|$, $\bS_k$ is a conditionally independent family given $\bS_{k-1}$. In particular, for $\alpha \in {\N}_K^G$, we have
\begin{equation*}\label{eq: independence from the rest final} S_\alpha \underset{\bS_{Restr'(\alpha,K)}}{\independent} \bS \backslash \bS_{[\alpha]_K}.
\end{equation*}
\end{cor}

\begin{proof} Fix $\alpha \in J_k$, and let $A$ be an arbitrary finite subset of $J_k \backslash \{\aa\}$. By Proposition \ref{prop:keytool}, we have
\begin{equation*} S_\alpha \underset{(S_{\alpha \wedge \beta}: \beta \in A)}{\independent} (S_\beta : \beta \in A),
\end{equation*}
and since $\sigma(S_{\alpha \wedge \beta}: \beta \in A) \subseteq \sigma(\bS_{Restr'(\alpha,K)}) \subseteq \sigma(S_\aa)$, by Lemma \ref{lem: towering} we have
\begin{equation*} S_\alpha \underset{\bS_{Restr'(\alpha,K)}}{\independent} (S_\beta : \beta \in A).
\end{equation*}
Since $A$ is arbitrary, we have
\begin{equation}\label{eq: independence from the rest1} S_\alpha \underset{\bS_{Restr'(\alpha,K)}}{\independent} \bS \backslash \bS_{[\alpha]_K},
\end{equation}
and since $\bS_{Restr'(\alpha,K)} \in \sigma(\bS_{k-1}) \subseteq \sigma(\bS \backslash \bS_{[\alpha]_K})$, again by Lemma \ref{lem: towering} we have
\begin{equation}\label{eq: independence from the rest2} S_\alpha \underset{\bS_{k-1}}{\independent} \bS \backslash \bS_{[\alpha]_K}.
\end{equation}
Since $\alpha$ is arbitrary, the proof is complete.
\end{proof}

\begin{proof}[Proof of Theorem \ref{thm:main2}] We build an induction to show that for all $k \leq |G|$, there exists an i.i.d. array of uniform random variables $(U_\alpha : \alpha \in J_k)$ such that \eqref{eq: main2} holds for all $\alpha \in J_k$. (The case $k=0$ is obvious.) Let us assume that there exists an i.i.d. array $\bW_{k-1}:=(W_\alpha : \alpha \in J_{k-1})$ of uniform random variables and a family of measurable functions $(f_C: C \in \AA_G, |C| \leq k-1) $ such that almost surely,
\begin{equation}\label{eq: inductive hypothesis 4.6}
S_\alpha = f_{\dom(\alpha)}(\bS_{Restr'(\alpha)}, W_\alpha), \alpha \in J_{k-1}.
\end{equation}

Fix $\alpha \in N_k$, where $N_k:=J_k \backslash J_{k-1}$. By \eqref{eq: independence from the rest1} from Corollary \ref{cor: independence given the previous}, we have
\begin{equation*}
S_\alpha \underset{\bS_{Restr'(\alpha,K)}}{\independent} \bS_{k-1}.
\end{equation*}
Thus, by Lemma \ref{lem: conditional independence randomization}, there exists a uniform random variable $V_\alpha$ independent of $\bS_{k-1}$ such that
\begin{equation}\label{eq: coding 1, 4.6} S_\alpha = f_\alpha(\bS_{Restr'(\alpha)}, V_\alpha)
\end{equation}
almost surely. By exchangeability, there exists an array of uniform random variables $\mathbf{\partial V}_k:=(V_\alpha : \alpha \in N_k)$, which are not necessarily independent, such that \eqref{eq: coding 1, 4.6} holds for every $\alpha \in N_k$, with the choice of $f_\alpha$ identical for all $\alpha$ defined on the same domain, which we denote by $f_{\dom(\aa)}.$

Now consider an array $\mathbf{\partial W}_k:=(W_\alpha : \alpha \in N_k) $ of i.i.d. uniform random variables, which are also independent of $\bW_{k-1}$, and for $\alpha \in N_k$ define
\begin{equation}\label{eq: replacing S, 4.6}
    S_\alpha':= f_\alpha (\bS_{Restr'(\alpha)}, W_\alpha).
\end{equation} Since we can replace \eqref{eq: inductive hypothesis 4.6} and \eqref{eq: replacing S, 4.6} into equations of the form \eqref{eq: main3}, we can combine them into a one-line expression of the form
\begin{equation}\label{eq: one line 4.6}
\bS'_k = F(\bW_k),
\end{equation}
for some function $F$ where $\bS'_k = \Big( \bS_{k-1}, (S'_\alpha : \alpha \in N_k) \Big)$ and $\bW_k=(W_\alpha : \alpha \in J_k)$. Note that $\bW_k = (\bW_{k-1}, \mathbf{\partial W}_k)$ is an i.i.d. array. On the other hand, we have the following properties:
\begin{itemize}
    \item $\P[S_\alpha \in \cdot | \bS_{k-1}]=\P[S'_\alpha \in \cdot | \bS_{k-1}]$ almost surely for all $\alpha \in N_k$ since both $W_\aa$ and $V_\aa$ are independent of $\bS_{k-1}$.
    \item $(S_\alpha : \alpha \in N_k)$ is a conditionally independent family given $\bS_{k-1}$ by Corollary \ref{cor: independence given the previous}.
    \item $(S'_\alpha : \alpha \in N_k)$ is a conditionally independent family given $\bS_{k-1}$ by construction.
\end{itemize}
By Lemma \ref{lem: identification} we have $$\bS_k=\Big( \bS_{k-1}, (S_\alpha : \alpha \in N_k) \Big) \buildrel d \over = \Big( \bS_{k-1}, (S'_\alpha : \alpha \in N_k) \Big).$$ Thus, with \eqref{eq: one line 4.6}, we can apply Lemma \ref{lem: coding lemma2} to obtain an array of i.i.d. uniform random variables $\bU_k:=(U_\alpha : \alpha \in J_k)$ such that
\begin{equation}\label{eq: one line 4.6, 2}
\bS_k = F(\bU_k)
\end{equation}
almost surely. Splitting \eqref{eq: one line 4.6, 2} back to individual equations of the form \eqref{eq: main2}, we obtain the desired representation for dimension $k$. Since $k \leq |G|$ is arbitrary, we have \eqref{eq: main2} for all $\alpha \in \Gamma_K^G$ by induction. 
\end{proof}
\begin{proof}[Proof of Theorem \ref{thm:main}] Let $\alpha \in \N_G^K \cap \Gamma_K^G$. Since $X_\alpha$ is $\sigma(\bX)$-measurable and is invariant under permutations fixing $\alpha$, we have $X_\alpha \in \FF_\alpha$ and hence $X_\alpha=h(S_\alpha)$ for some measurable function $h$. By inserting \eqref{eq: main3}, we obtain \eqref{eq: main1} by identifying $I_K^G$ modulo $\buildrel K \over \sim$ with $\Gamma_K^G$. 

Let $U_\alpha:=U_{\gamma_\alpha}$ for $\aa \in \N_K^G$ (not necessarily in $\Gamma_K^G$ and $\bU:=(U_\aa : \aa \in I_K^G).$ To obtain \eqref{eq: main1} for all $\alpha \in \N_K^G$, it suffices to show that we can choose $\bU$ in a way such that $(\bX, \bU)$, or equivalently $(\bS, \bU)$, is $(G,K)$-exchangeable. Indeed, if $(\bX,\bU)$ is exchangeable, since for $\beta \in \N_K^G$ there exists $\alpha \in \N_G^K \cap \Gamma_K^G$ such that $\tau \alpha = \beta$ with $\tau \in Z_K(S_\N^{\rtimes G})$ by transitivity of the group action, \eqref{eq: main1} holds if we replace $\alpha$ with $\beta = \tau \alpha$ by exchangeability.

By Transfer Lemma \ref{lem: coding lemma}, there exists a family of measurable functions $(\phi_C : C \in \AA_G)$ such that for any $\alpha \in \Gamma_K^G$ and any uniform random variable $V$ independent of $S_\alpha$, we have 

\begin{equation}\label{eq: rev 1}
(S_\alpha, U_\alpha) \buildrel d \over = (S_\alpha, \phi_{\dom(\alpha)}(S_\alpha, V)).\footnote{We can choose $\phi_C$ to depend only on the domain by exchangeability of $\bS$.}
\end{equation}

 Let $\mathbf{V}=(V_\alpha : \alpha \in I_K^G)$ be an array of uniform random variables independent of $\bS$, where $V_\aa = V_{\gamma_\aa}$ for all $\aa \in I_K^G$ and different components are all independent. Then, by \eqref{eq: one line 4.6, 2} and \eqref{eq: rev 1}, we have $$S_\alpha = f_G(\bS_{Restr'(\aa)}, U'_\alpha)$$ almost surely for all $\alpha \in \Gamma_K^G$ where $U'_\alpha = \phi_{\dom(\alpha)}(S_\aa, V_\aa).$ (Note that $\bS_{Restr'(\aa)} \in \sigma(S_\aa)$.)

Since $\mathbf{V}$ is an i.i.d. array independent of $\bS$, we have
\begin{equation}\label{eq: 4.5, 1} U_\alpha' \underset{\bS}{\independent} (U'_\beta: \beta \in \Gamma_K^G \backslash \{\alpha\})
\end{equation}
for all $\alpha \in \Gamma_K^G$. Also, since $(U_\bb : \bb \in \Gamma_K^G)$ is an independent family and $\bS \backslash \bS_{[\aa]_K} \in \sigma(U_\bb : \bb \in \Gamma_K^G \backslash \{\aa\})$, we have
\begin{equation}\label{eq: 4.5, 2} U_\aa \underset{\bS \backslash \bS_{[\aa]_K}}{\independent} (U_\bb : \bb \in \Gamma_K^G \backslash \{\aa\}).
\end{equation}
Since $S_\aa \in \sigma(\bS\backslash S_\aa, U_\aa)$, by Lemma \ref{lem: towering} we obtain
\begin{equation}\label{eq: 4.5, 3} U_\aa \underset{\bS}{\independent} (U_\bb : \bb \in \Gamma_K^G \backslash \{\aa\}).
\end{equation}
Therefore by \eqref{eq: 4.5, 1} and \eqref{eq: 4.5, 3}, both $\Big((S_\alpha, U_\alpha) : \alpha \in \Gamma_K^G\Big)$ and $\Big((S_\alpha, U'_\alpha) : \alpha \in \Gamma_K^G\Big)$ are conditionally independent family given $\bS$.
Since $(S_\aa, U_\aa) \buildrel d \over = (S_\aa, U_\aa')$ and both are conditionally independent of $\bS$ given $S_\aa$, by Lemma \ref{lem: identification} we have
\begin{equation}\label{eq: 4.5, 4} (\bS, U_\aa) \buildrel d \over = (\bS, U_\aa')
\end{equation}
for all $\alpha \in \Gamma_K^G$. Therefore, by Lemma \ref{lem: identification}, we have
\begin{equation}\label{eq: 4.5, 5} \Big(\bS, (U_\aa : \aa \in I_K^G)\Big) \buildrel d \over = \Big(\bS, (U'_\aa : \aa \in I_K^G)\Big)
\end{equation}
where $U'_\aa := U'_{\gamma_\aa}$. Thus, the relations \eqref{eq: main2} still hold even if we replace $\bU$ by $\bU':=(U'_\aa : \aa \in I_K^G)$. Since $\bS$ and $\mathbf{V}$ are exchangeable and independent of each other, $(\bS, \mathbf{V})$ is exchangeable. Thus, by Lemma \ref{lem: joining exchangeability 1}, $(\bS, \bU')$ is exchangeable.
\end{proof}
\section{Proof of Proposition 2.10}

Let $\FF_\alpha = \sigma(S_\aa)$, as in Section \ref{section: symmetry}. 
\begin{lem}\label{lem: tail coding} For $\alpha \in I_K^G$, $n \in \N$, let $$\FF_\alpha^n := \sigma(X_\beta: \beta \in {\N}_K^G, \text{ there exists } C \in \AA_G \text{ such that } \beta|_C \in Restr(\alpha,K) \text{ and } \beta(v) \geq n \text{ for all } v \notin C).$$
Then, $\FF_\alpha = \underset{n \geq 1}{\bigcap} \FF_\alpha^n$. 
\end{lem}

Let us introduce some notations to be used in the proof. For $\aa,\bb \in I^G_K$ and $v \in G$, let $$A_\aa(v):=\Big\{\alpha(w) : C_w\backslash\{w\}=C_v\backslash\{v\}, \ w \buildrel K \over \sim v \Big\},$$
and for $B \subseteq \N$ we define the injection $\tau_B: \N \to \N$ as $$\tau_B(n) := \begin{cases} \min \{k > n : k \notin B\}, & n \notin B, \\ n, & n \in B. \end{cases}$$
In other words, $\tau_B$ fixes the numbers in $B$ and shifts the rest to the increasing direction.

We define $\rho_\aa$ to be an injective map of $I_K^G$ to itself as 
$$\rho_\aa(\beta)(v)=\begin{cases} \tau_{A_\aa(w)}(\beta(v)), & \bb|_{C_v\backslash \{v\}} \buildrel K \over \sim \aa|_{C_w\backslash \{w\}} , \\ \bb(v)+1, & \text{ otherwise.} \end{cases}$$
The choice of $w$ is irrelevant because we will always have $C_w\backslash\{w\}=C_u\backslash\{u\}$ for any $u \in \dom(\aa)$ such that $\bb|_{C_v\backslash \{v\}} \buildrel K \over \sim \aa|_{C_u \backslash \{u\}}$ (otherwise we do not have $\aa \in \N_K^G$). The evaluation divides into three cases:
\begin{enumerate}
    \item If $\kappa(\beta|_{C_v}) = \aa|_{C_w}$ for some $w \in \dom(\aa)$ and $\kappa \in K$ (or equivalently, $\beta|_{C_v} \in Restr(\aa,K)$), then $\rho_\aa(\bb)(v)=\bb(v).$
    \item If for some $w \in \dom(\aa)$ and $\kappa \in K_v$ we have $\kappa(\beta|_{C_v})(u) = \aa|_{C_w}(u)$ for all $u \in C_w \backslash \{w\}$ but $\kappa (\bb |_{C_v})(w) \neq \aa (w)$, we let $\rho_\aa(\bb)(v)=\bb(v)+\ell$, where $\ell$ is the smallest positive integer such that $\beta_{\ell,v}|_{C_v} \notin Restr(\aa,K)$ and
    $$\beta_{\ell,v}(u):=\beta(u) + \ell \mathbbm{1}_{\{u=v\}}.$$
    \item Otherwise, we have $\rho_\aa(\bb)(v)=\bb(v)+1$.
\end{enumerate}

It is straightforward (but tedious) to check that $\rho_\aa$ is an injection satisfying \eqref{eq: sep DAG exchangeability} in Definition \ref{def: sep DAG exch}, and that it commutes with $K$. Thus, any restriction of $\rho_\aa$ to a finite subset of $\N_G^K$ can be extended to an element in $Z_K(S_\N^{\rtimes G})$. Thus, by Kolmogorov extension theorem, the array $\bX$ law-invariant under the action of $\rho_\aa$. 

Also note that $\rho_\aa$ fixes $\aa$.

\begin{proof}[Proof of Lemma \ref{lem: tail coding}]  For any $\beta \in {\N}_K^G$, letting $C:=D_{\beta, \alpha}$, we have $\bb|_C \in Restr(\aa,K) $ (see the remark right before Proposition \ref{prop:keytool}). For $v \in C$, we have $\rho_\aa(\bb)(v)=\bb(v)$. On the other hand, for $v \notin C$, one can see that $\rho_\aa(\beta)(v) \geq \beta(v)+1$. Thus, we have $\rho_\aa(\bX) \in \FF_\alpha^2.$ Since $D_{\rho_\aa(\bb),\aa}=D_{\bb,\aa}$ for all $\bb \in \N_K^G$, we can similarly show that $\rho_\aa^n(\bb)(\bX) \in \FF_\aa^{n+1}$.

So, for arbitrary $E \in \FF_\alpha$, acting $\rho_\aa^n$ on both sides of the inclusion $E \in \sigma(\bX)$ we obtain $E \in \sigma(\rho_\aa^n(\bX)) \subseteq \FF_\alpha^{n+1}$. This shows that $\FF_\alpha \subseteq \underset{n \leq 1}{\bigcap} \FF_\alpha^n$.

To prove the converse, consider the collection $\TT_n$ of all finite permutations $\tau \in Z_K(S_\N^{\rtimes G})$ such that
\begin{enumerate}
    \item $\tau$ fixes $\alpha$. 
    \item for all $\beta$, $\tau(\beta)(v)=\beta(v)$ whenever there exists $u \preceq v$ such that $\beta(u) > n$.
\end{enumerate}
Then, as $n \to \infty$, the collection $\TT_n$ eventually contains all finite permutations in $Z_K(S_\N^{\rtimes G})$ fixing $\alpha$, and $\FF_\alpha^n$ is invariant under the action of $\TT_n$. Therefore, we have $\underset{n \leq 1}{\bigcap} \FF_\aa^n \subseteq \FF_\aa$.
\end{proof}

\begin{proof}[Proof of Proposition \ref{prop:keytool}]
We use induction on the number of indices $n$. The case $n=1$ is obvious.

Consider $\rho:=\rho_{\alpha_1}$ which is defined as in the proof of Lemma \ref{lem: tail coding}. As we have seen in the proof of Lemma \ref{lem: tail coding}, acting $\rho$ on $\beta$ fixes the values on $D_{\beta, \aa_1}$, which is by definition the values in the domain of $\beta \wedge \alpha_1$, and shifts all the values outside $D_{\beta, \aa_1}$ by at least +1. By exchangeability, we have $$\big(S_{\aa_1},...,S_{\aa_n}\big) \buildrel d \over = \big(S_{\aa_1}, \rho^k(S_{\aa_2}),..., \rho^k(S_{\aa_n})\big).$$ By Lemma \ref{lem: inclusive conditional indep}, we have
\begin{equation}\label{eq: key1}
S_{\alpha_1} \underset{\rho^k(S_{\aa_2}),..., \rho^k(S_{\aa_n})}{\independent} S_{\aa_2},...,S_{\aa_n}
\end{equation}
for all $k\in\N$. Since the indices generating $\rho^k(\FF_{\aa_j})$ are contained in the collection of the indices generating $\FF_{\aa_1 \wedge \aa_j}^k$, we have $\rho^k(S_{\aa_j}) \in \FF_{\aa_1 \wedge \aa_j}^k$. Thus, for all $k \geq 2$ we have
\begin{equation}\label{eq: key2}
S_{\alpha_1} \underset{\FF_{\alpha_1 \wedge \alpha_j}^k : 2 \leq j \leq n}{\independent} S_{\aa_2},...,S_{\aa_n}
\end{equation}
due to Lemma \ref{lem: inclusive conditional indep}. By the inductive hypothesis, $S_{\alpha_2},...,S_{\alpha_n}$ are conditionally independent given $(S_{\aa_i \wedge \aa_j} : 2 \leq i \neq j \leq n)$. Thus, by Lemma \ref{lem: tail coding} and Lemma \ref{lem: tail emersion}, the $\sigma$-field  generated by $\FF^k_{\aa_1 \wedge \aa_2},..., \FF^k_{\aa_1 \wedge {\aa_n}} $ decreases to some $\GG \subseteq {\sigma(\FF_{\aa_i \wedge \aa_j}} : 1 \leq i \neq j \leq n)$. Applying the backward martingale convergence to \eqref{eq: key2} as $k \to \infty$, we obtain
\begin{equation*}
    S_{\alpha_1} \underset{\GG}{\independent} S_{\aa_2},...,S_{\aa_n},
\end{equation*}
and since each $\FF_{\aa_i \wedge \aa_j}$ is a sub-$\sigma$-field of $S_{\alpha_i}$, by Lemma \ref{lem: towering} we obtain
\begin{equation}\label{eq: key3}
    S_{\alpha_1} \underset{(\FF_{\alpha_i \wedge \alpha_j} : 1 \leq i \neq j \leq n)}{\independent} S_{\aa_2},...,S_{\aa_n}.
\end{equation}
Since $S_{\alpha_2},...,S_{\alpha_n}$ are conditionally independent given $(\FF_{\aa_i \wedge \aa_j} : 2 \leq i \neq j \leq n)$ and for each $j \geq 2$, $\FF_{\aa_1 \wedge \aa_j}$ is a sub-$\sigma$-field of $\FF_{\aa_j}$, again by Lemma \ref{lem: towering} we have that $S_{\alpha_2},...,S_{\alpha_n}$ are conditionally independent given $(\FF_{\aa_i \wedge \aa_j} : 1 \leq i \neq j \leq n)$. Combining this with \eqref{eq: key3}, we obtain the desired result.
\end{proof}
\appendix
\section{Supplementary Lemmas}\label{app: supp}

Elementary results that we use in the main text are introduced in this section. All are standard results and frequently used in exchangeability theory. For those results without proofs we have added references where one can find the proofs. We note again that the richness of the probability space is always assumed.

\begin{lem}\label{lem: towering} Let $\GG$, $\FF_1,\FF_2,\FF_3$ be $\sigma$-fields satisfying $\FF_1 \subseteq \FF_2 \subseteq \FF_1 \vee \FF_3$ and
$$ \FF_3 \underset{\FF_1}{\independent} \GG$$
for some $\sigma$-field $\GG$. Then, we have
$$\FF_3 \underset{\FF_2}{\independent} \GG.$$
\end{lem}
\begin{proof} The proof is straightforward from the towering property of conditional expectations.
\end{proof}
\begin{lem}[Transfer Lemma: Theorem 6.10, \cite{kallenberg2002foundations}]\label{lem: coding lemma} Let $X,Y$ be random elements in a Borel space. Then,
\begin{enumerate}
    \item For all $X' \buildrel d \over = X$, there exists a measurable function $f$ such that whenever $W$ is a uniform random variable independent of $X'$, then $Y':=f(X',W)$ satisfies $(X,Y) \buildrel d \over = (X',Y').$
    \item There exist measurable functions $h$ and $g$ such that whenever $W$ is a uniform random variable independent of $X$ and $Y$, $V:=h(X,Y,W)$ is a uniform random variable independent of $X$ satisfying $Y=g(X,V)$ almost surely. 
\end{enumerate}
\end{lem}

\begin{lem}[Corollary 6.11, \cite{kallenberg2002foundations}]\label{lem: coding lemma2} Let $X$, $Y$ be Borel-valued random variables such that $X \buildrel d \over = f(Y)$ for some measurable function $f$. Then, there exists a random variable $Y' \buildrel d \over = Y$ such that $X=f(Y')$ almost surely.
\end{lem} 
\begin{lem}[Lemma 1.3, \cite{kallenberg2006probabilistic}]\label{lem: inclusive conditional indep} Let $X,Y,Z$ be random variables such that $(X,Y) \buildrel d \over = (X,Z)$ and $\sigma(Y) \subseteq \sigma(Z)$. Then, $X \underset{Y}{\independent} Z$.
\end{lem}
\begin{lem}[Proposition 5.13, \cite{kallenberg2002foundations}]\label{lem: conditional independence randomization} Let $X,Y,Z$ be random elements, where $X$ lies in a Borel space. Then, $X$ is conditionally independent of $Z$ given $Y$ if and only if there exists a measurable function $f$ and a uniform random variable $U$ independent of $Y,Z$ such that $X=f(Y,U)$ almost surely.
\end{lem}

\begin{lem}\label{lem: joining exchangeability 1} Let $H$ be a group acting measurably on Borel spaces $\XX$ and $\YY$, and let $\mu$ be an $H$-invariant probability measure on $\XX$, that is, $x$ is $H$-exchangeable under $\mu$. Let $\phi: \XX \to \YY$ be a measurable function. If $\phi(\tau x)=\tau \phi(x)$ $\mu$-almost surely for all $\tau \in H$, then $(x, \phi(x))$ is $H$-exchangeable under $\mu$.
\end{lem}
\begin{proof} $\mu[x \in A, \phi(x) \in B]=\mu[\tau x \in A, \phi(\tau x) \in B] = \mu[\tau x \in A, \tau \phi(x)].$
\end{proof}

\begin{lem}\label{lem: identification} Let $(X_i : i \in I)$, $(Y_i : i \in I)$ be a family of random variables with a countable index set $I$. For a random variable $S$, assume that the following are true:
\begin{itemize}
    \item $(S, X_i) \buildrel d \over = (S, Y_i)$. Equivalently, $\P[X_i \in \cdot | S] = \P[Y_i \in \cdot | S]$ almost surely.
    \item Given $S$, Both $(X_i: i \in I)$ and $(Y_i : i \in I)$ are conditionally independent families.
\end{itemize}
Then, we have $(S, X_i : i \in I) \buildrel d \over = (S, Y_i : i \in I).$
\end{lem}
\begin{proof} Without loss of generality, let $I=\N$. Then for $n \in \N$ and bounded measurable functions $f_1,...,f_n$,
\begin{align*} \E[f_1(X_1) \cdots f_n(X_n)|S]  & = \E[f_1(X_1)|S] \cdots \E[f_n(X_n)|S] \\ & = \E[f_1(Y_1)|S] \cdots \E[f_n(Y_n)|S] = \E[f_1(Y_1) \cdots f_n(Y_n)|S].
\end{align*}
\end{proof}

\begin{lem}\label{lem: tail emersion}
For each $n \in \N$, let $(\FF_k^n : k\in \N)$ be a sequence of decreasing $\sigma$-fields with $\FF^n:=\underset{k \geq 1}{\bigcap}\FF_k^n$. Assume that given $\GG$, the family $(\FF_1^n: n\in \N)$ is independent.

Then, $\FF:=\underset{k \in \N}{\bigcap}(\underset{n \in \N}{\vee} \FF_k^n) $ is a sub-$\sigma$-field of $\GG \underset{n \in \N}{\vee} \FF^n$.\footnote{Without the conditional independence, we cannot guarantee the result. Consider two sequences of random variables $\bX=(X_n : n \in \N)$ and $\bY=(Y_n : n \in \N)$, and let $P$ be a uniform random variable. Suppose that given $P$, $\bX$ and $\bY$ are independent i.i.d. sequences, where $\P[X_1 = 1 | P] = 1-\P[X_1 = -1| P]=P$ and $\P[Y_1 = 1 | P] = 1-\P[Y_1 = -1| P]=1/2$. Let $Z_n := X_n Y_n$. Then, $\bY$ and $\mathbf{Z}$ are i.i.d. sequences of random variables independent of $P$, and hence their tail $\sigma$-fields are trivial. However, since the tail $\sigma$-field of the joint sequence $(\bY, \mathbf{Z})$ recovers $P$, and hence it is not equal to the join of the tail $\sigma$-fields of the components.} In particular, if $(\FF_1^n: n\in \N)$ are unconditionally independent, then $\FF= \underset{n \in \N}{\vee}\FF^n.$
\end{lem}
\begin{proof}
For each $n \in \N$ choose an event $A_n \in \FF_1^n$, and let $A \in \GG$. Then,
\begin{align*} \P \Big[ A \underset{j \leq N}{\bigcap} A_j | \GG, \FF_k^n : n \in \N \Big] =  \mathbbm{1}_A \underset{j \leq N}{\prod} \P[A_j | \FF_k^j, \GG]
\end{align*}
by conditional independence. By backward martingale convergence the right hand side converges to a $\GG \underset{n \in \N}{\vee} \FF^n$-measurable random variable as $k \to \infty$. Since the collection of all the events of the form $A  \underset{j \leq N}{\bigcap} A_j$ is a $\pi$-system generating $\GG \underset{n \in \N}{\vee} \FF_1^n$, we can use $\pi$-$\lambda$ arguments to show that for any $\GG \underset{n \in \N}{\vee} \FF_1^n$-measurable event $B$, we have $$ \underset{k \to \infty}{\lim} \P[B | \GG, \FF_k^n : n \in \N ] \in \GG \underset{n \in \N}{\vee} \FF^n.$$
This shows that $\GG \underset{n \in \N}{\vee} \FF_k^n$ converges to $\GG \underset{n \in \N}{\vee} \FF^n$ as $k \to \infty$. Since $\underset{n \in \N}{\vee}\FF_k^n \subseteq \GG \underset{n \in \N}{\vee} \FF_k^n$ , we can conclude that $$\underset{k\in \N}{\bigcap} (\underset{n \in \N}{\vee}\FF_k^n) \subseteq \underset{k\in \N}{\bigcap} (\GG \underset{n \in \N}{\vee} \FF_k^n) = \GG \underset{n \in \N}{\vee} \FF^n.$$ The last statement is obvious since we always have $\underset{n \in \N}{\vee}\FF^n \subseteq \underset{n \in \N}{ \vee} \FF_k^n$ for each $k \in \N$, which implies that $\underset{n \in \N}{\vee}\FF^n \subseteq \underset{k\in \N}{\bigcap}(\underset{n \in \N}{ \vee} \FF_k^n)=\FF$.
\end{proof}

\section{Proof of Proposition 2.8}\label{app: sym array}

For $\alpha \in I_K^G$, let us write $\gamma_\alpha$ for the representative of $[\alpha]_K$.

For each $C \in \AA_G$, choose $\gamma_C \in \Gamma_K^G$ with $\dom(\gamma_C)=C$, if there is any. Since $\FF_{\gamma_C}$ is countably generated, there exists a Borel-valued random variable $S_{\gamma_C}$ such that $\sigma(S_{\gamma_C})=\FF_{\gamma_C}.$ \footnote{Any countably generated $\sigma$-field can be generated by a random variable taking values on a Borel space. The obvious choice of the generator for $\FF=\sigma(A_k : k \in \N)$ would be $X=\underset{k \in \N }{\sum} 2^{-k} \mathbbm{1}_{A_k}.$} Since $S_{\gamma_C}$ is $\sigma(\bX)$-measurable, there exists a measurable function $f_C$ such that $f_C(\bX)=S_{\gamma_C}$ almost surely. For any other $\gamma \in \Gamma_K^G$ with $\dom(\gamma_C)=C$, choose $\tau \in Z_K(S_\N^{\rtimes G})$ such that $\tau \gamma = \gamma_C$ and let $$S_\gamma := f_C(\tau \bX).$$ 
Note that the choice of $\tau$ is irrelevant. For $\aa \in I_K^G$, we let
$$S_\aa : = S_{\gamma_\aa}.$$

 Let $\alpha, \beta \in I_K^G$, $\tau \in Z_K(S_\N^{\rtimes G})$ with $\alpha = \tau \beta$. We claim that $$\phi(\tau \bX) = S_\bb$$
 whenever $\phi(\bX)=S_\aa$. The left hand side represents the action of $\tau$ as we regard $S_\aa$ as an $\sigma(\bX)$-measurable random element, while on the right hand side $\tau$ acts on $\sigma(\bS)$-measurable random elements. Once we have that these actions are identical, Lemma \ref{lem: joining exchangeability 1} implies that the array $\bS$ constructed this way satisfies the desired properties.

Let $\dom(\gamma_\beta)=D$. Then, we have $S_{\gamma_\beta}=f_D(\lambda \bX)$, for $\lambda \in Z_K(S_\N^{\rtimes G})$ such that $\gamma_D = \lambda \gamma_\beta$. Considering the way we have chosen the representatives, $\gamma_\alpha$ and $\gamma_\beta$ are defined on the same domain, and hence for some $\rho$ we have $\gamma_D = \rho \gamma_\alpha $ and thus $S_{\gamma_\alpha}=f_D(\rho \bX).$

For $v \in C$, let $\kappa_1,\kappa_2 \in K$ be local isomorphisms such that $$\kappa_1 (\bb|_{C_{\kappa_1(v)}}) = \gamma_\bb|_{C_v}, \kappa_2 (\gamma_\aa|_{C_{\kappa_2\kappa_1(v)}}) = \aa|_{C_{\kappa_1(v)}}.$$ 
Then, $\rho \tau \lambda^{-1}$ fixes $\gamma_D$ at $v \in D$ since
\begin{align*} \rho \tau \lambda^{-1} (\gamma_D|_{C_v}) & = \rho \tau (\gamma_\bb|_{C_v}) = \rho \tau \kappa_1 (\bb|_{C_{\kappa_1(v)}}) \\ & = \rho \kappa_1 \tau (\bb|_{C_{\kappa_1(v)}}) = \rho \kappa_1 (\alpha|_{C_{\kappa_1(v)}}) \\ & = \rho \kappa_1 \kappa_2 (\gamma_\aa|_{C_{\kappa_2\kappa_1(v)}}) = \kappa_1 \kappa_2 \rho (\gamma_\aa|_{C_{\kappa_2\kappa_1(v)}}) \\ & = \kappa_1 \kappa_2 (\gamma_D|_{C_{\kappa_2\kappa_1(v)}}) = \gamma_D|_{C_v}.
\end{align*}
Since $v \in D$ is arbitrary, $\rho \tau \lambda^{-1}$ fixes $\gamma_D$. Thus, by definition of $\bS$, we have $$ f_D(\bX) = f_D(\rho \tau \lambda^{-1} \bX)$$ almost surely. By exchangeability the equation holds almost surely if we replace $\bX$ by $\lambda \bX$, and hence $$ S_{\gamma_\bb}=f_D(\lambda \bX)=f_D(\rho \tau \bX)$$ almost surely. 
\subsection*{Acknowledgments}
The author was supported by the National Research Foundation of Korea (NRF-2017R1A2B2001952), a National Research Foundation of Korea grant funded by the Korean Government MSIT (NRF-2019R1A5A1028324)
\bibliographystyle{alpha}
\bibliography{bib}

\newcommand{\etalchar}[1]{$^{#1}$}
\begin{thebibliography}{LOGR12}

\bibitem[Ald81]{aldous1981representations}
David~J Aldous.
\newblock Representations for partially exchangeable arrays of random
  variables.
\newblock {\em Journal of Multivariate Analysis}, 11(4):581--598, 1981.

\bibitem[AP14]{austin:panchenko:2014}
Tim Austin and Dmitry Panchenko.
\newblock A hierarchical version of the de {F}inetti and {A}ldous-{H}oover
  representations.
\newblock {\em Probability Theory and Related Fields}, 159(3-4):809--823, 2014.

\bibitem[Aus08]{austin2008exchangeable}
Tim Austin.
\newblock On exchangeable random variables and the statistics of large graphs
  and hypergraphs.
\newblock {\em Probability Surveys}, 5:80--145, 2008.

\bibitem[BRT19]{bloem2019probabilistic}
Benjamin Bloem-Reddy and Yee~Whye Teh.
\newblock Probabilistic symmetry and invariant neural networks.
\newblock {\em arXiv preprint arXiv:1901.06082}, 2019.

\bibitem[BZSL13]{bruna2013spectral}
Joan Bruna, Wojciech Zaremba, Arthur Szlam, and Yann LeCun.
\newblock Spectral networks and locally connected networks on graphs.
\newblock {\em arXiv preprint arXiv:1312.6203}, 2013.

\bibitem[CW16]{cohen2016group}
Taco Cohen and Max Welling.
\newblock Group equivariant convolutional networks.
\newblock In {\em International conference on machine learning}, pages
  2990--2999, 2016.

\bibitem[DF29]{de1929funzione}
Bruno De~Finetti.
\newblock Funzione caratteristica di un fenomeno aleatorio.
\newblock In {\em Atti del Congresso Internazionale dei Matematici: Bologna del
  3 al 10 de settembre di 1928}, pages 179--190, 1929.

\bibitem[DF37]{de1937prevision}
Bruno De~Finetti.
\newblock La pr{\'e}vision: ses lois logiques, ses sources subjectives.
\newblock In {\em Annales de l'institut Henri Poincar{\'e}}, volume~7, pages
  1--68, 1937.

\bibitem[DJ08]{diaconis2008graph}
Persi Diaconis and Svante Janson.
\newblock Graph limits and exchangeable random graphs.
\newblock {\em Rendiconti di Matematica}, 28:33--61, 2008.

\bibitem[FP{\etalchar{+}}12]{fortini2012predictive}
Sandra Fortini, Sonia Petrone, et~al.
\newblock Predictive construction of priors in bayesian nonparametrics.
\newblock {\em Brazilian Journal of Probability and Statistics},
  26(4):423--449, 2012.

\bibitem[Hof08]{hoff2008modeling}
Peter Hoff.
\newblock Modeling homophily and stochastic equivalence in symmetric relational
  data.
\newblock In {\em Advances in neural information processing systems}, pages
  657--664, 2008.

\bibitem[Hoo79]{hoover1979relations}
Douglas~N Hoover.
\newblock Relations on probability spaces and arrays of random variables.
\newblock {\em Preprint, Institute for Advanced Study, Princeton, NJ}, 2, 1979.

\bibitem[HS55]{hewitt1955symmetric}
Edwin Hewitt and Leonard~J Savage.
\newblock Symmetric measures on cartesian products.
\newblock {\em Transactions of the American Mathematical Society},
  80(2):470--501, 1955.

\bibitem[JLSY20]{jung2020generalization}
P~Jung, J~Lee, S~Staton, and H~Yang.
\newblock A generalization of hierarchical exchangeability on trees to directed
  acyclic graphs.
\newblock {\em Annales Henri Lebesgue}, 2020.

\bibitem[Kal89]{kallenberg1989representation}
Olav Kallenberg.
\newblock On the representation theorem for exchangeable arrays.
\newblock {\em Journal of Multivariate Analysis}, 30(1):137--154, 1989.

\bibitem[Kal92]{kallenberg1992symmetries}
Olav Kallenberg.
\newblock Symmetries on random arrays and set-indexed processes.
\newblock {\em Journal of Theoretical Probability}, 5(4):727--765, 1992.

\bibitem[Kal02]{kallenberg2002foundations}
Olav Kallenberg.
\newblock {\em Foundations of modern probability}.
\newblock Springer, 2002.

\bibitem[Kal05]{kallenberg2006probabilistic}
Olav Kallenberg.
\newblock {\em Probabilistic Symmetries and Invariance Principles}.
\newblock Springer, 2005.

\bibitem[LOGR12]{lloyd2012random}
James Lloyd, Peter Orbanz, Zoubin Ghahramani, and Daniel~M Roy.
\newblock Random function priors for exchangeable arrays with applications to
  graphs and relational data.
\newblock In {\em Advances in Neural Information Processing Systems}, pages
  998--1006, 2012.

\bibitem[OR14]{orbanz2014bayesian}
Peter Orbanz and Daniel~M Roy.
\newblock Bayesian models of graphs, arrays and other exchangeable random
  structures.
\newblock {\em IEEE transactions on pattern analysis and machine intelligence},
  37(2):437--461, 2014.

\bibitem[SYA{\etalchar{+}}17]{statonetal2017pps}
Sam Staton, Hongseok Yang, Nathanael~L. Ackerman, Cameron Freer, and Daniel~M
  Roy.
\newblock Exchangeable random process and data abstraction.
\newblock In {\em Workshop on Probabilistic Programming Semantics (PPS 2017)},
  2017.

\end{thebibliography}
\end{document}